\documentclass[a4paper, 12pt]{article}

\newif\ifAMS
\IfFileExists{amssymb.sty}
  {\AMStrue\usepackage{amssymb}}{}
\usepackage{latexsym}
\usepackage{tikz}
\usepackage{amsthm}
\usepackage{amsmath}

\usepackage{graphicx,textcomp,booktabs,amsmath,pst-node}
\usepackage{mathptmx,courier}
\usepackage[scaled]{helvet}

\newtheorem{thm}{Theorem}[section]
\newtheorem{cor}[thm]{Corollary}
\newtheorem{prop}{Proposition}[section]
\newtheorem{lem}[thm]{Lemma}
\newtheorem{rem}{Remark}[section]

\theoremstyle{definition}
\newtheorem{defn}{Definition}[section]

\numberwithin{equation}{section}

\usepackage{amsmath}

\usepackage[ansinew]{inputenc}
\usepackage[T1]{fontenc}
\usepackage{mathptmx,courier}
\usepackage[scaled]{helvet}
\usepackage[arrow, matrix, curve]{xy}
\usepackage{hyperref}
\begin{document}
\bigskip
\bigskip
\bigskip
\bigskip
\bigskip

\title{Phases of stable representations of quivers}

\bigskip
\bigskip
\bigskip
\bigskip

\author{Magnus Engenhorst\footnote{engenhor@math.uni-bonn.de}\\
Mathematical Institute, University of Bonn\\
Endenicher Allee 60, 53115 Bonn, Germany\\[5mm]}

\maketitle

\begin{abstract}
We consider stable representations of non-Dynkin quivers with respect to a central charge. On one condition the existence of a stable representation with self-extensions implies the existence of infinitely many stables without self-extensions. In this case the phases of the stable representations approach one or two limit points. In particular, the phases are not dense in two arcs.      
\end{abstract}

\bigskip

\section{Introduction}

We study stable representations of acyclic quivers\footnote{Throughout this paper the quivers are finite and connected.} with respect to a central charge. Our main tool is the \textit{categorical mutation method} developed in \cite{18,28} as mathematical counterpart to the mutation method in physics based on Seiberg duality \cite{50}. The BPS states of some supersymmetric quantum field theories can be identified with stable representations of quivers with potential \cite{1,2}. In \cite{50} the authors use the mutation method to calculate the stable representations for a given central charge based on physical arguments. Their work suggests in the presence of a stable higher spin state there are infinitely many stable representations approaching an accumulation ray, i.e. a central charge occupied by an infinite set of states. We consider the case of acyclic quivers, i.e. the superpotential $W$ is trivial ($W=0$). The higher spin states correspond to representations of the quiver with self-extensions (see section 3.1 of \cite{3}). Thus we are lead to study the stable representations of a quiver $Q$ in the presence of a stable representation $V$ with $Ext^{1}_{Q}(V,V)\neq 0$. Another motivation is the study of the generalized Kronecker quiver in \cite{210}. We need the following condition on the central charge: the stable modules without self-extensions have a unique phase, i.e. all other stable modules have a different phase.  We call such a central charge \textit{rigid}. Non-isomorphic stable modules with self-extensions can have the same phase and this will happen alle the time. We have the following result which is basically Proposition \ref{prop2} in the case of the category of finite-dimensional left $kQ$-modules $mod-kQ$ where $kQ$ is the path algebra of an acyclic quiver $Q$:

\begin{thm}
\label{11}
Let $Q$ be an acyclic quiver and let $k$ be an algebraically closed field. Given a rigid central charge $Z:K(mod-kQ)\rightarrow\mathbb{C}$ and a stable module $V\in mod-kQ$ with $Ext^{1}_{Q}(V,V)\neq 0$, then there are infinitely many stable objects and their phases approach a limit point from the left and a limit point from the right. In particular, the phases of the stable objects with respect to the central charge are not dense in two intervals in $(0,1]$ on the left/right of $\arg\ Z(V)$. The stable objects with phases in these intervals have no self-extensions. 
\end{thm}     

In the case of an Euclidean quiver there are infinitely many stable postprojective respectively preinjective stable indecomposables approaching the ray $Z(V)$ for the stable module $V$ with self-extensions from right respectively left (Prop. \ref{qeuclidean}). The situation for a wild quiver is more complicated. The infinitely many exceptional modules described in Theorem \ref{11} can contain infinitely many regular modules. The wild case is considered in Proposition \ref{wild}.\\ In section 2 we review Bridgeland stability conditions on triangulated categories which underlies the categorical mutation method. The latter is introduced in detail in section 3. In section 4 we enlarge upon Theorem \ref{11} for Euclidean quivers and in section 5 for wild quivers.

\section{Stability conditions on triangulated categories}

Let $\mathcal{D}$ be a triangulated category. We denote by $K(\mathcal{D})$ the Grothendieck group of $\mathcal{D}$.

\begin{defn}~\cite{5}
\label{bridgeland}
A \textit{stability condition} on a triangulated category $\mathcal{D}$ consists of a group homomorphism $Z:K(\mathcal{D})\rightarrow \mathbb{C}$ called the \textit{central charge} and of full additive subcategories $\mathcal{P}(\phi)\subset\mathcal{D}$ for each $\phi\in \mathbb{R}$, satisfying the following axioms:
\begin{enumerate}
\item if $0\neq E\in\mathcal{P}(\phi)$, then $Z(E)=m(E)exp(i\pi\phi)$ for some $m(E)\in \mathbb{R}_{>0}$;
\item $\forall \phi\in\mathbb{R}, \mathcal{P}(\phi+1)=\mathcal{P}(\phi)\left[1\right]$;
\item if $\phi_{1}>\phi_{2}$ and $A_{j}\in\mathcal{P}(\phi_{j})$, then $Hom_{\mathcal{D}}(A_{1},A_{2})=0;$
\item for $0\neq E\in\mathcal{D}$, there is a finite sequence of real numbers $\phi_{1}>\cdots>\phi_{n}$ and a collection of triangles 

$$0=\xymatrixcolsep{7mm}\xymatrix{E_{0}\ar[rr] && E_{1}\ar[ld]\ar[rr] && E_{2}\ar[ld]\\ \
			& A_{1}\ar @{-->}[lu]			& & A_{2}\ar @{-->}[lu]}\rightarrow\cdots\rightarrow\xymatrixcolsep{7mm}\xymatrix{ E_{n-1}\ar[rr] && E_{n}\ar[ld]\\			
			 & A_{n}\ar @{-->}[lu]}=E$$
       
with $A_{j}\in\mathcal{P}(\phi_{j})$ for all j. 
\end{enumerate}
\end{defn} 

We recall some results of \cite{5}. The subcategory $\mathcal{P}(\phi)$ is Abelian and its nonzero objects are said to be \textit{semistable} of phase $\phi$ for a stability condition $\sigma=(Z,\mathcal{P})$. We call its simple objects \textit{stable}. The objects $A_{i}$ in Definition $\ref{bridgeland}$ are called semistable factors of E with respect to $\sigma$. For any interval $I\subset\mathbb{R}$ we define $\mathcal{P}(I)$ to be the extension-closed subcategory of $\mathcal{D}$ generated by the subcategories $\mathcal{P}(\phi)$ for $\phi\in I$. The full subcategory $\mathcal{P}(>\phi)$ is a bounded t-structure in $\mathcal{D}$ with heart $\mathcal{P}((\phi,\phi+1])$. A stability condition is \textit{locally-finite} if there exists some $\epsilon>0$ such that for all $\phi\in\mathbb{R}$ each quasi-Abelian subcategory $\mathcal{P}((\phi-\epsilon,\phi+\epsilon))$ is of finite length. In this case the subcategory $\mathcal{P}(\phi)$ is of finite length. We denote by $Stab(\mathcal{D})$ the set of locally-finite stability conditions. It is a topological space.\\

\begin{defn} A \textit{central charge} (or \textit{stability function}) on an Abelian category $\mathcal{A}$ is a group homomorphism $Z:K(\mathcal{A})\rightarrow\mathbb{C}$ such that for any nonzero $E\in\mathcal{A}$, $Z(E)$ lies in the upper halfplane
\begin{align}
\label{halfplane}
\overline{\mathbb{H}}:=\left\{r\cdot exp(i\pi\phi)| 0<\phi\leq 1,r\in\mathbb{R}_{>0}\right\}\subset\mathbb{C}.
\end{align}
\end{defn}

Every object $E\in\mathcal{A}$ has a phase $0<\phi(E)\leq 1$ such that $$Z(E)=r\cdot exp(i\pi\phi(E))$$ with $r\in\mathbb{R}_{>0}$. We say a nonzero object $E\in\mathcal{A}$ is \textit{(semi)stable} with respect to the central charge $Z$ if every proper subobject $0\neq A\subset E$ satisfies $\phi(A)< \phi(E)$ ($\phi(A)\leq \phi(E)$).\\
We will frequently use the following fact: If $E_1$ is semistable of phase $\phi_1$ and $E_2$ is semistable of phase $\phi_2$ with respect
to a central charge on $\mathcal{A}$, then $Hom_{\mathcal{A}}(E_1,E_2)=0$, if $\phi_1>\phi_2$.\\

The central charge $Z$ has the Harder-Narasimhan (HN) property if every nonzero object $E\in \mathcal{A}$ has a finite filtration $$0=E_{0}\subset E_{1}\subset \ldots \subset E_{n-1}\subset E_{n}=E$$ where the semistable factors $F_{j}=E_{j}/E_{j-1}$ fulfill $$\phi(F_{1})>\phi(F_{2})>\ldots >\phi(F_{n}).$$

The next result and its proof are crucial for the following sections. 
\begin{prop}~\cite{5}
\label{prop5}
To give a stability condition on a triangulated category $\mathcal{D}$ is equivalent to giving a bounded t-structure on $\mathcal{D}$ and a central charge on its heart which has the Harder-Narasimhan property.
\end{prop}
\begin{proof}Given a heart $\mathcal{A}$ of a bounded t-structure on $\mathcal{D}$ and a central charge with HN property we define the subcategories $\mathcal{P}(\phi)$ to be the semistable objects of $\mathcal{A}$ of phase $\phi\in(0,1]$ together with the zero-objects of $\mathcal{D}$ and continue by the rule $\mathcal{P}(\phi+1)=\mathcal{P}(\phi)[1]$.\\ Conversely, given a stability condition $\sigma=(Z,\mathcal{P})$ on a triangulated category $\mathcal{D}$ the full subcategory $\mathcal{A}=\mathcal{P}((0,1])$ is the heart of a bounded t-structure on $\mathcal{D}$. Identifying the Grothendieck groups $K(\mathcal{A})$ and $K(\mathcal{D})$ the central charge $Z:K(\mathcal{D})\rightarrow\mathbb{C}$ defines a central charge on $\mathcal{A}$. The semistable objects of the categories $\mathcal{P}(\phi)$ are the semistable objects of $\mathcal{A}$ with respect to this central charge.
\end{proof}   

\begin{defn}
A heart $\mathcal{A}\subset\mathcal{D}$ of a t-structure of a triangulated category $\mathcal{D}$ is called \textit{algebraic} if (i) all its objects have finite length, i.e. there are no infinite chains of inclusions or quotients for all objects of $\mathcal{A}$, and (ii) there are only finitely many simple objects. We call a heart $\mathcal{A}\subset\mathcal{D}$ \textit{rigid} if all its simple objects $S$ fulfill $Ext_{\mathcal{A}}^{1}(S,S)=0$.
\end{defn}

The central charge on an algebraic heart has automatical the HN property. Thus Proposition \ref{prop5} reduces the construction of stability conditions with algebraic heart $\mathcal{A}=\mathcal{P}((0,1])$ to the definition of a central charge on $\mathcal{A}\subset \mathcal{D}$ by choosing a complex number in $\overline{\mathbb{H}}$ for every simple object. Let $S_{1},\ldots,S_{n}$ be the simple objects of $\mathcal{A}$. Then the subset $Stab(\mathcal{A})$ of $Stab(\mathcal{D})$ consisting of stability conditions with heart $\mathcal{A}$ is isomorphic to $\mathbb{H}^{n}$.

\begin{defn}\cite{230}
A \textit{torsion pair} in an Abelian category $\mathcal{A}$ is a pair of full subcategories $(\mathcal{T}, \mathcal{F})$ satisfying
\begin{enumerate}
\item $Hom_{\mathcal{A}}(T,F)=0$ for all $T\in\mathcal{T}$ and $F\in\mathcal{F}$;
\item every object $E\in\mathcal{A}$ fits into a short exact sequence
\begin{align}
0\longrightarrow T\longrightarrow E\longrightarrow F\longrightarrow 0 \nonumber
\end{align}
for some pair of objects $T\in\mathcal{T}$ and $F\in\mathcal{F}$.
\end{enumerate} 
\end{defn}  

The objects of $\mathcal{T}$ are called \textit{torsion} and the objects of $\mathcal{F}$ are called \textit{torsion-free}.\\
 
Given a torsion pair in an heart of a bounded t-structure of a triangulated category $\mathcal{D}$ define new hearts of a bounded t-structure on $\mathcal{D}$ \cite{20}. This construction is called \textit{tilting}. The new hearts in this case after tilting are
\begin{align}
L_{S}(\mathcal{A})&=\left\{E\in \mathcal{D}| H^{i}(E)=0 \text{ for } i\notin \lbrace0,1\rbrace, H^{0}(E)\in\mathcal{F}, H^{1}(E)\in \left\langle S\right\rangle\right\}, \nonumber \\
R_{S}(\mathcal{A})&=\left\{E\in \mathcal{D}| H^{i}(E)=0 \text{ for } i\notin \lbrace-1,0\rbrace, H^{-1}(E)\in\left\langle S\right\rangle, H^{0}(E)\in\mathcal{T} \right\}. \nonumber
\end{align}
Here $H^{i}(E)$ are the cohomology objects of $E$ with espect to the initial bounded t-structure. For the details see \cite{20}. $L_{S}(\mathcal{A})$ (respectively $R_{S}(\mathcal{A})$) is called \textit{the left} (respectively \textit{the right}) \textit{tilt of $\mathcal{A}$ at the simple $S$}. $S[-1]$ is a simple object in $L_{S}(\mathcal{A})$ and if this heart is again of finite length we have $R_{S[-1]}L_{S}(\mathcal{A})=\mathcal{A}$. Similarly, if $R_{S}(\mathcal{A})$ has finite length, we have $L_{S[1]}R_{S}(\mathcal{A})=\mathcal{A}$.\\

We are interested in the case of a simple object of $\mathcal{A}$ leaving the upper halfplane. We have the following crucial result:

\begin{prop}
\label{lemma} 
(\cite{10}, Lemma 5.5) Let $\mathcal{A}\subset \mathcal{D}$ be an algebraic heart of a bounded t-structure on $\mathcal{D}$. Let $\sigma\in Stab(\mathcal{A})$ be  a stability conditions such that the simple object $S$ has phase $1$ and all other simples in $\mathcal{A}$ have phases in $(0,1)$. Then there is a neighbourhood of $\sigma$ lying in $Stab(\mathcal{A})\cup Stab(L_{S}(\mathcal{A}))$.
\end{prop}

\section{Mutation method}

In this section we review the \textit{categorical mutation method}. Most results of this work are based on Proposition \ref{prop2} proved in this section.

\begin{defn}
Let $\mathcal{A}\subset\mathcal{D}$ be an algebraic heart of a bounded t-structure of $\mathcal{D}$. We say we can \textit{tilt $\mathcal{A}$ indefinitely} if the left-tilt $L_{S}(\mathcal{A})$ at any simple object of $\mathcal{A}$ is again of finite length. The tilt of $L_{S}(\mathcal{A})$ at any simple object of $L_{S}(\mathcal{A})$ is again of finite length and so on.
\end{defn}

Recall the notion of a central charge from the last section.

\begin{defn} Let $\mathcal{A}$ be an Abelian category. We call a central charge $Z:K(\mathcal{A})\rightarrow\mathbb{C}$ \textit{discrete} if different stable objects of $\mathcal{A}$ have different phases. 
\end{defn}

Let $\mathcal{A}\subset\mathcal{D}$ be the heart of a bounded t-structure of $\mathcal{D}$ of finite length with finitely many simple objects. Let us consider the $n$ simple objects $S_{1},\ldots, S_{n}$ of $\mathcal{A}$. We assume there is a simple object $S_{i}$ that is \textit{left-most}, i.e. whose phase in $(0,1]$ is the bigger than the phases of the other simple objects with respect to the central charge $Z:K(\mathcal{A})\rightarrow\mathbb{C}$. First we assume that the central charge of $S_{1},\ldots, S_{n}$ lie in the upper half-plane above the real axis. Then we rotate the central charge $Z:K(\mathcal{A})\rightarrow\mathbb{C}$ by rotating the complex numbers $Z(S_{1}),\ldots, Z(S_{n})$ a bit counter-clockwise. This corresponds to deforming stability conditions in $Stab(\mathcal{A})$ within $Stab(\mathcal{A})$ until we reach a stability condition $\sigma$ on the boundary of $Stab(\mathcal{A})$ corresponding to a central charge where the left-most simple object $S_i$ of $\mathcal{A}$ lies on the negative real axis and all other simple objects still lie in the upper half-plane. Now we are in the situation of Proposition \ref{lemma} and thus there is a neighborhood of $\sigma$ that lies in $Stab(\mathcal{A})\cup Stab(L_{S_{i}}(\mathcal{A}))$. If we rotate a bit further the corresponding stability conditions all lie in $Stab(L_{S_{i}}(\mathcal{A}))$. If a simple object of the tilted heart is left-most we can repeat the same process for the tilted heart $L_{S_{i}}(\mathcal{A})$ and proceed further (if possible).\\

Let $\mathcal{A}\subset\mathcal{D}$ be an algebraic heart of a bounded t-structure of $\mathcal{D}$. We assume we can tilt $\mathcal{A}$ indefinitely. We summarize the steps of the \textit{(categorical) mutation method}:

\begin{enumerate}
\item Start with a stability condition in $Stab(\mathcal{A})$ and deform it within $Stab(\mathcal{A})$ by rotating the central charge $Z:K(\mathcal{D})\rightarrow\mathbb{C}$ counter-clockwise. 
\item If the left-most simple object $S$ leaves the upper half-plane tilt at this left-most simple $S$.
\item Deform within $Stab(L_{S}(\mathcal{A}))$ by rotating the central charge further till the left-most object of $L_{S}(\mathcal{A})$ leaves the upper half-plane and tilt $L_{S}(\mathcal{A})$ at this simple object.
\item Repeat this procedure (if possible).  
\end{enumerate}

\begin{rem}The mutation method is similary defined for rotating the central charge clock-wise with right-tilts at the right-most simple objects.
\end{rem}

Given a discrete central charge with finitely many stable objects the following theorem is the main result. We include a proof since we will use the same arguments to study stable objects for more general central charges in the next section.

\begin{thm}\cite{18}
\label{prop}
Let $\mathcal{A}\subset\mathcal{D}$ be an algebraic heart of a bounded t-structure of $\mathcal{D}$ with simple objects $S_{1},\ldots, S_{n}$. We assume we can tilt $\mathcal{A}$ indefinitely and we have given a discrete central charge $Z:K(\mathcal{A})\rightarrow\mathbb{C}$ with finitely many stable objects. Then the left-most simple objects of hearts appearing in the mutation method are the stable objects of $\mathcal{A}$. In the order of decreasing phase they give a sequence of simple tilts from $\mathcal{A}$ to $\mathcal{A}[-1]$.
\end{thm}
\begin{proof}
We repeat the mutation algorithm described above until we accomplish a rotation by $\pi$. Under the assumptions in the theorem this is possible. For the details see \cite{18}. In the mutation method we always tilt at objects in $\mathcal{A}$ as can be seen as follows: The first tilt is at a simple object in $\mathcal{A}$. Then the simple objects in the first tilted heart are in $\mathcal{A}$ or in $\mathcal{A}[-1]$. Since we tilt at the left-most simple of a heart we tilt next at an object in $\mathcal{A}$. This is because all objects in $\mathcal{A}[-1]$ will have a smaller phase than objects in $\mathcal{A}$ since we rotate counter-clockwise. We prove that the simple objects of the hearts appearing in the mutation method are in $\mathcal{A}$ or in $\mathcal{A}[-1]$ by induction: The simple objects of a tilted heart are in $\mathcal{A}$ or in $\mathcal{A}[-1]$: Let us assume that a heart $\mathcal{A}'$ appearing in the mutation method is the left-tilt of $\mathcal{A}$ with respect to some torsion pair $(\mathcal{T},\mathcal{F})$ in $\mathcal{A}$. Then the simple objects of $\mathcal{A}'$ lie in $\mathcal{F}\subset\mathcal{A}$ or in $\mathcal{T}[-1]\subset\mathcal{A}[-1]$. We tilt next at an object $S'$ in $\mathcal{A}$. Thus $S'\in\mathcal{F}$ and $\left\langle S'\right\rangle\subset \mathcal{F}$. By Lemma 3.2 in \cite{190} the simple objects of the heart obtained by tilting $\mathcal{A}'$ at $S'$ is the left-tilt of some torsion pair of $\mathcal{A}$.\\ 

As long as a simple object of a heart appearing in the mutation method lies in $\mathcal{A}$ and thus its central charge lies in $\overline{\mathbb{H}}$ we have not accomplished a rotation by $\pi$. The final heart $\mathcal{A}'$ obtained in the mutation method contains only simple objects in $\mathcal{A}[-1]$. We have therefore $\mathcal{A}'\subset\mathcal{A}[-1]$ and this implies $\mathcal{A}'=\mathcal{A}[-1]$. If all simple objects are in $\mathcal{A}[-1]$ we are in the final heart.\\

All left-most simple objects of a heart appearing in the mutation method are stable objects in $\mathcal{A}$. The phases of all other stable objects in $\mathcal{A}$ are smaller than the phase of the first left-most simple object $S$ since we chose a discrete central charge. By the definition of the left-tilt at a simple object all stable objects except the left-most simple remain in the first tilt of $\mathcal{A}$ since there are no homomorphisms between $S$ and the other stable objects. In the first tilted heart the phases of the stable objects of $\mathcal{A}$ are equal or smaller than the new left-most simple object. If the phase of a stable object of $\mathcal{A}$ is equal to this left-most simple they are the same since we chose a discrete central charge. Otherwise the stable object remains in the next tilted heart and so on. Since we rotate the central charge further and further every stable object of $\mathcal{A}$ has to appear as a left-most simple of a heart. Therefore we tilt in the mutation method at all stable objects of $\mathcal{A}$. For every central charge, we tilt at all initial simple objects $S_{1},\ldots, S_{n}$ since these are stable for any central charge.  
\end{proof}

Let $k$ be an algebraically closed field. Let $Q=(Q_{0},Q_{1})$ be a finite connected quiver with set of vertices $Q_{0}$ and set of arrows $Q_{1}$. We denote by $kQ$ its path algebra, i.e. the algebra with basis given by all paths in $Q$ and product given by composition of paths. Let $mod-kQ$ be the category of finite-dimensional left modules over $kQ$ and $D^{b}(mod-kQ)$ be its bounded derived category.

\begin{prop}
Let $Q$ be a Dynkin quiver. Let $\mathcal{A}$ be a heart of a bounded t-structure in $D^{b}(mod-kQ)$. The stable objects with respect to a discrete central charge on $\mathcal{A}$ in the order of decreasing phase define a sequence of simple tilts from $\mathcal{A}$ to $\mathcal{A}[-1]$.
\end{prop} 
\begin{proof}
By Proposition 6.1 in \cite{200} any heart of a bounded t-structure in $D^{b}(mod-kQ)$ is algebraic. By Lemma 3.13 in \cite{210} the set of phases of stable objects is finite and thus there are only finitely many stable objects in $mod-kQ$. Now the Proposition follows from Theorem \ref{prop}. 
\end{proof}

Let $\widehat{kQ}$ be the completion of $kQ$ at the ideal generated by the arrows of $Q$. We consider the quotient of $\widehat{kQ}$ by the subspace $[\widehat{kQ},\widehat{kQ}]$ of all commutators. It has a basis given by the cyclic paths of $Q$ (up to cyclic permutation). For each arrow $a\in Q_{1}$ the cyclic derivative is the linear map from the quotient to $\widehat{kQ}$ which takes an equivalence class of a path $p$ to the sum $$\sum_{p=uav}vu$$ taken over all decompositions $p=uav$. An element $$W\in HH_{0}(\widehat{kQ})=\widehat{\frac{\widehat{kQ}}{[\widehat{kQ},\widehat{kQ}]}}$$ is called a \textit{(super)potential} if it does not involve cycles of length $\leq 2$.

\begin{defn}\cite{80}
\label{jacobialgebra}
Let $(Q,W)$ be a quiver $Q$ with potential $W$. The \textit{Jacobi algebra} $\mathfrak{J}(Q,W)$ is the quotient of $\widehat{kQ}$ by the twosided ideal generated by the cyclic derivatives $\partial_{a}W$: $$\mathfrak{J}(Q,W):=\widehat{kQ}/(\partial_{a}W, a\in Q_{1}).$$
\end{defn}

We denote the category of finite-dimensional right modules over $\mathfrak{J}(Q,W)$ by $mod-\mathfrak{J}(Q,W)$. Given a quiver with potential $(Q,W)$ we can define a differential graded algebra $\Gamma=\Gamma((Q,W))$, the Ginzburg algebra of $(Q,W)$. \cite{90} Let $D_{fd}(\Gamma)$ be the finite-dimensional derived category of the Ginzburg algebra $\Gamma$ \cite{100}. 

\begin{thm}
\label{theorem}
Let $(Q,W)$ be a 2-acyclic quiver $Q$ with non-degenerate\footnote{Non-degenerate in the sense of \cite{80}.} potential $W$ such that we have a discrete central charge on $\mathcal{A}=mod-\mathfrak{J}(Q,W)$ with finitely many stable objects. Then the sequence of stable objects of $\mathcal{A}$ in the order of decreasing phase defines a sequence of simple tilts from $\mathcal{A}$ to $\mathcal{A}[-1]$ in $D_{fd}(\Gamma)$. Moreover, $\mathfrak{J}(Q,W)$ is finite-dimensional.  
\end{thm}

\begin{cor}
\label{corolary}
In this case for a stable object $S$ of $\mathcal{A}=mod-\mathfrak{J}(Q,W)$ we have $Hom_{\mathcal{A}}(S,S)=k$ and $Ext_{\mathcal{A}}^{1}(S,S)=0$.
\end{cor}
\begin{proof}
The stable objects appear as left-most simple objects of hearts of bounded t-structures in $D_{fd}(\Gamma)$ that are equivalent to the category of finite-dimensional modules over the Jacobi algebra of a quiver with potential $(Q',W')$ (mutation-equivalent to $(Q,W)$). This equivalence is induced by a $k$-linear triangle equivalence of $D_{fd}(\Gamma((Q',W'))$ and $D_{fd}(\Gamma((Q,W))$. The result follows from the fact that the simple objects of $mod-\mathfrak{J}(Q',W')$ have the claimed properties.  
\end{proof}

In the setting of Theorem \ref{theorem} there are finitely many stable objects and all come without self-extensions. Next we want allow infinitely many stable objects and stables with self-extensions.

\begin{rem}
$mod-\mathfrak{J}(Q,W)$ for a non-degenerate potential $W$ is representation-finite if and only if $Q$ is mutation-equivalent to a Dynkin quiver.\cite{88} Since stable modules are indecomposable there are finitely many stables for any central charge in this case.
\end{rem}
 
\begin{rem}
The full subcategory $\mathcal{P}(\phi)$ of semistable objects in $mod-\mathfrak{J}(Q,W)$ of phase $\phi$ together with the zero objects is an Abelian subcategory of finite length, i.e. there is a Jordan-H\"older filtration of every semisimple object in the simple objects given by the stable objects of phase $\phi$. It follows from Corollary \ref{corolary} that these subcategories are semisimple with a unique simple object. Thus in this case the definition of a discrete central charge is equivalent to the stronger definition of B. Keller in \cite{140} that demands the latter property to be true. 
\end{rem}

\begin{defn}
Let $\mathcal{A}\subset\mathcal{D}$ be a heart of a bounded t-structure of a triangulated category $\mathcal{D}$. We call a central charge $Z:K(\mathcal{A})\rightarrow\mathbb{C}$ \textit{rigid} if the stable objects $S$ with $Ext_{\mathcal{A}}^{1}(S,S)=0$ have a unique phase, i.e. all other stable objects have a different phase.
\end{defn}

\begin{prop}
\label{prop2}
Let $\mathcal{A}\subset\mathcal{D}$ be an algebraic heart of a bounded t-structure of a triangulated category $\mathcal{D}$ that we can tilt indefinitely. We assume the initial heart $\mathcal{A}$ and all hearts obtained by a finite sequence of simple tilts are rigid. Let $Z:K(\mathcal{A})\rightarrow\mathbb{C}$ be a rigid central charge on $\mathcal{A}$. Further we assume there is stable object $V$ with $Ext_{\mathcal{A}}^{1}(V,V)\neq 0$. Then there are infinitely many stable objects and their phases approach a limit point from the left and a limit point from the right.\footnote{These limit points can coincide.} In particular, the phases of the stable objects with respect to the central charge are not dense in two intervals in $(0,1]$ on the left/right of $\arg\ Z(V)$. The stable objects with phases in these intervals have no self-extensions. 
\end{prop}                  
\begin{proof}
The assumptions imply that all simple objects have a unique phase. Thus there is a left-most simple and we can apply the mutation algorithm. We use the same arguments as in the proof of Theorem \ref{prop}: We rotate counter clock-wise till the first left-most simple $S$ just leaves the upper halfplane $\overline{\mathbb{H}}$. We end up in a heart $\mathcal{A}'$ given by the left tilt at $S$ of $\mathcal{A}$. All stable objects of $\mathcal{A}$ except $S$ are objects of $\mathcal{A}'$. In particular, the stable object $V$. There is an unique left-most simple without self-extensions in this heart that is a stable object in $\mathcal{A}$ and we can repeat the mutation algorithm. The object $V$ with self-extensions can never be left-most by the assumptions. Therefore there must in every step of the mutation method lie a simple left-most on the left of the ray $Z(V)$ and we can repeat the procedure indefinitely.\\
The same arguments applied to the mutation method with clock-wise rotation and right tilts show there are infinitely many stable objects right to the ray $Z(V)$.
\end{proof}

For two hearts $\mathcal{A}_{1},\mathcal{A}_{2}$ with associated bounded t-structures $\mathcal{F}_{1}, \mathcal{F}_{2}\subset \mathcal{D}$ we say $\mathcal{A}_{1}\leq \mathcal{A}_{2}$ if and only if $\mathcal{F}_{2}\subset\mathcal{F}_{1}$.

\begin{cor}
Under the conditions of Proposition \ref{prop2} the Abelian category $\mathcal{A}$ has infinitely many torsion classes $(\mathcal{T},\mathcal{F})$ whose associated hearts $\mathcal{A}'=\left\langle \mathcal{F},\mathcal{T}[-1]\right\rangle$ of bounded t-structures of $D^{b}(mod-kQ)$ fulfill $\mathcal{A}[-1]\leq\mathcal{A}'\leq\mathcal{A}$. 
\end{cor}
\begin{proof}
Given the heart $\mathcal{A}$ of a bounded t-structure $\mathcal{G}\subset \mathcal{D}$ the t-structure is the extension-closed subcategory $$\mathcal{G}=\left\langle \mathcal{A},\mathcal{A}[1],\mathcal{A}[2],\ldots\right\rangle.$$ By the proof of Theorem \ref{prop} every heart $\mathcal{A}'$ appearing in the mutation method is a left-tilt of the initial heart $\mathcal{A}$ at same torsion pair $(\mathcal{T},\mathcal{F})$ in $\mathcal{A}$. We have $\mathcal{F}\subset\mathcal{A}'$ and $\mathcal{T}\subset \mathcal{A}'[1]$. Since the torsion pair $(\mathcal{T},\mathcal{F})$ generates $\mathcal{A}$ we have $\mathcal{A}'\leq\mathcal{A}$. Further $\mathcal{F}$ and $\mathcal{T}[-1]$ lie in the t-structure $$\left\langle \mathcal{A}[-1],\mathcal{A},\mathcal{A}[1],\ldots\right\rangle$$ associated to the heart $\mathcal{A}[-1]$. Thus we have $\mathcal{A}[-1]\leq\mathcal{A}'$. 
\end{proof}

\section{Stable representations of Euclidean quivers}

Let $k$ be an algebraically closed field and $Q$ an acyclic quiver. In this and the next section we study the case that we have a central charge on $mod-kQ$ with at least one stable object $V$ with $Ext_{Q}^{1}(V,V)\neq 0$. By Theorem \ref{prop} and Corollary \ref{corolary} there can be no discrete central charge with finitely many stables with this property. In fact, there will be infinitely many stable objects with the same class $[S]$ in the Grothendieck group of $mod-kQ$.\\

Let us first consider the Kronecker quiver 
\begin{align}  
\label{kroneckerquiver}
\xymatrix{1 \ar@<+.7ex>[r]\ar@<-.7ex>[r] & 2}.
\end{align}

Let us denote by $S_{1}$ and $S_{2}$ the simple representations associated with the two vertices. If the phase of $S_{2}$ is greater than the phase of $S_{1}$, the simples are the only stable objects. If the phase of $S_{1}$ is greater than the phase of $S_{2}$ the stable objects are precisely all postprojective and preinjective representations together with the representations in the $\mathbb{P}_{k}^{1}$-family with dimension vector $(1,1)$:
\begin{align}
\label{kronecker}
  \xymatrix{k \ar@<+.7ex>[r]^{\lambda_0}\ar@<-.7ex>[r]_{\lambda_1} & k}
  \end{align}
with $(\lambda_0:\lambda_1)\in\mathbb{P}_{k}^{1}$. We see in this case there are infinitely many stable objects lying on a ray in the upper half plane. The phases of the postprojective and preinjective indecomposables approach this accumulation ray from right and left (see figure 1).\\

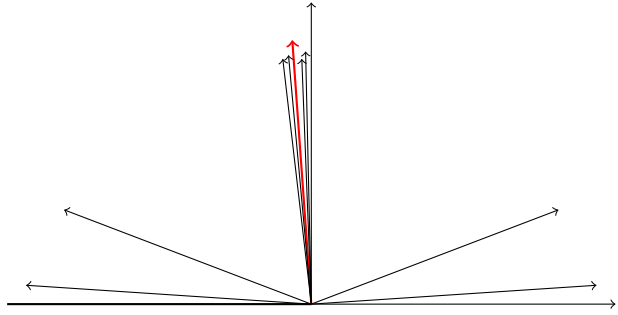
\begin{figure}
\begin{center}
\begin{tikzpicture}[scale=2.5]
    \draw [->] (0,0) -- (1.6,0);
    \draw [->] (0,0)--(0,1.6);
    \draw [thick] (-1.6,0) -- (0,0);
    \draw [->] (0,0) --(1.5,0.1);
    \draw[->] (0,0)--(-1.5,0.1);
    \draw[red,thick][->] (0,0)--(-0.1,1.4);
    \draw[->] (0,0)--(-0.15,1.3);
    \draw[->] (0,0)--(-0.05,1.3);
    \draw[->] (0,0)--(-0.12,1.32);
    \draw[->] (0,0)--(-0.03,1.34);
    \draw[->] (0,0)--(1.3,0.5);
    \draw[->] (0,0)--(-1.3,0.5);
\end{tikzpicture}
\end{center}
\caption{Phases for the Kronecker quiver with $\phi(S_1)>\phi(S_2)$. In red the central charge of the stables with dimension vector $(1,1)$. The central charges of the stable postprojective/preinjective indecomposables lie to the right/left of this ray.}
\end{figure}

Let $Q=(Q_0,Q_1)$ be an acyclic quiver with head and tail map $h,t:Q_1\rightarrow Q_0$. The \textit{Euler form} is the bilinear form on the Grothendieck group $K(mod-kQ)=\mathbb{Z}Q_{0}$ given by 
\begin{align*}
\left\langle\ ,\ \right\rangle:\mathbb{Z}Q_{0}\times \mathbb{Z}Q_{0}&\longrightarrow \mathbb{Z}\\
(x,y)&\longmapsto \sum_{i\in Q_0}x_i y_i-\sum_{a\in Q_1}x_{t(a)}y_{h(a)}
\end{align*}

for $x=(x_i)_{i\in Q_0}, y=(y_i)_{i\in Q_0}$.\\

For two modules $X,Y\in mod-kQ$ we have the important formula 
\begin{align}
\label{formula}
\left\langle \underline{dim}\ X,\underline{dim}\ Y\right\rangle=dim\ Hom_{Q}(X,Y)-dim\ Ext_{Q}^{1}(X,Y).
\end{align}

To the Euler form associated is the \textit{Tits form}, the quadratic form on $\mathbb{Z}Q_{0}$ $$q_{Q}(x)=\sum_{i\in Q_0}x_i^2-\sum_{a\in Q_1}x_{t(a)}x_{h(a)}.$$ The \textit{radical} of $q$ is $rad(q)=\left\{x\in\mathbb{Z}Q_0|q(x)=0\right\}$. In the case $Q$ is Dynkin we have $rad(q)=0$ since the Tits form is positive definite. For an Euclidean quiver, i.e. an acyclic quiver with underlying graph an extended Dynkin graph, we have $rad(q)=\mathbb{Z}\delta$ for some $\delta\in \mathbb{Z}Q_0$ with $\delta_i>0$ for all $i\in Q_0$.\\

Let $S_1,\ldots, S_n$ be the simple modules of $mod-kQ$ where $n=\#Q_0$, $P(1),\ldots, P(n)$ its projective covers and $I(1),\ldots, I(n)$ its injective envelopes. An indecomposable module $X$ is called \textit{postprojective} if there is an integer $m\geq 0$ and $i\in Q_0$ such that $X\cong \tau^{-m}P(i)$, it is called \textit{preinjective} if there is an integer $m\geq 0$ and $i\in Q_0$ such that $X\cong \tau^{m}I(i)$ and it is called \textit{regular} if it is neither postprojective nor preinjective. Here $\tau$ is the Auslander-Reiten translation.

\begin{figure}
\begin{center}
\begin{tikzpicture}[scale=1.3]
    \draw (0,0) -- (0,2) -- (2,2); 
    \draw (0,0)--(2,0);
    \node at (1,1) {postprojective};
    \draw (2.2,0) -- (2.2,2) -- (4.2,2) --(4.2,0)--(2.2,0); 
    \node at (3.2,1) {regular};
    \draw (4.4,2)--(6.4,2);
    \draw (4.4,0)--(6.4,0)--(6.4,2);
    \node at (5.4,1) {preinjective};
\end{tikzpicture}
\end{center}
\caption{Components of the Auslander-Reiten quiver.}
\end{figure}
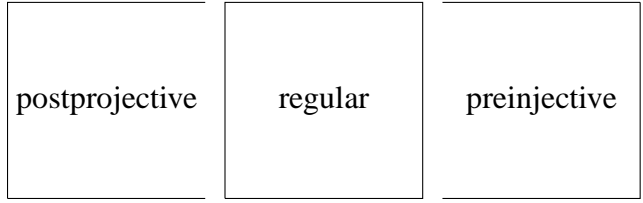

\begin{thm}\cite{220}
\label{dlabringel}
Let $Q$ be an Euclidean quiver and $X$ an indecomposable module of $Q$. Then $X$ is
\begin{enumerate}
\item[(i)] postprojective if and only if $\left\langle \delta, \underline{dim}\ X\right\rangle<0$,
\item[(ii)] preinjective if and only if $\left\langle \delta, \underline{dim}\ X\right\rangle>0$,
\item[(iii)] regular if and only if $\left\langle \delta, \underline{dim}\ X\right\rangle=0$.
\end{enumerate}
\end{thm}

The linear map $\left\langle\delta ,\ \right\rangle$ is called the \textit{Dlab-Ringel defect}.\\

The \textit{Coxeter transformation} $\Phi:\mathbb{Z}Q_0\rightarrow\mathbb{Z}Q_0$ is defined by $\Phi(\underline{dim}\ P(i))=-\underline{dim}\ I(i)$ for all $i\in Q_0$. As a matrix one has $\Phi=-(C^{-1})^{T}C\in GL(n,\mathbb{Z})$ where $C=(c_{ij})$ is the \textit{Cartan matrix} given by $c_{ij}=Hom_{Q}(P(i),P(j))$. The Euler form can be written in the form $$\left\langle x,y\right\rangle=x^{T}C^{-1}y.$$ For an indecomposable module $X$ that is not projective we have $${dim}(\tau X)=\Phi(\underline{dim}\ X).$$

Let $Q$ be an Euclidean quiver. Then its \textit{set of roots} is $\left\{x\in \mathbb{Z}Q_0|q(x)\leq 1\right\}$. $x\in \mathbb{Z}Q_0$ is a \textit{real} root if $q(x)=1$ and a \textit{imaginary} root if $q(x)=0$. The set of imaginary roots is $\left\{r\delta|r\in\mathbb{Z}\setminus\left\{0\right\}\right\}$ where $\delta$ was described above.\\

The stable modules $S$ in $mod-kQ$ with respect to a central charge are indecomposable with $Hom_{Q}(S,S)=k$. A module $X$ with $Hom_{Q}(X,X)=k$ is called a \textit{brick}. A brick is indecomposable. We call an indecomposable module $X$ exceptional if $Ext_{Q}^{1}(X,X)=0$. An exceptional module is a brick.   

\begin{lem}
\label{achscheisse}
Let $Q$ be an Euclidean quiver and $X\in mod-kQ$ be a brick. If $Ext_{Q}^{1}(X,X)=0$, then $\underline{dim}\ X$ is a real root, if $Ext_{Q}^{1}(X,X)\neq 0$, then $\underline{dim}\ X$ is a imaginary root.
\end{lem}
\begin{proof}
Let $X$ be a brick with $Ext_{Q}^{1}(X,X)\neq 0$. By formula (\ref{formula}) we have $$q_Q(\underline{dim}\ X)=\left\langle \underline{dim}\ X,\underline{dim}\ X\right\rangle=dim\ Hom_{Q}(X,X)-dim\ Ext_{Q}^{1}(X,X)=0$$ since $q_Q$ is positive semi-definite for a Euclidean quiver $Q$. The other case follows similary.
\end{proof}

The exceptional modules for the (generalized) Kronecker quiver are exactly the postprojective and preinjective modules. In general, there are exceptional modules in the regular component. \cite{250} An example is the following representation of the $\tilde{A}_2$-quiver: $$
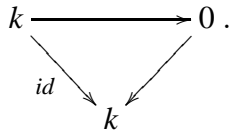
\xymatrix{k \ar[rr] \ar[rd]_{id}&& 0\ar[ld]\\
       & k}.$$

\begin{lem}
\label{achscheisse2}
Let $Q$ be an Euclidean quiver. There are at most finitely many regular modules which are exceptional.
\end{lem}
\begin{proof}
Let $Q$ have $n$ vertices. By consequence (3) in §9 of \cite{240} for a regular module $X$ which is a brick we have $\delta-\underline{dim}\ X\in\mathbb{N}^{n}$. Thus there are only finitely many classes $\alpha\in K(mod-kQ)$ with $\alpha=\underline{dim}\ X$ for a regular brick $X$. Since $\underline{dim}\ X$ is a positive real root for an exceptional module $X$ there is exactly one indecomposable module $X$ with class $\underline{dim}\ X$. This proves the lemma.
\end{proof}

The next Lemma is the first result on the phases of stable representations of quivers.

\begin{lem}
\label{achscheisse3}
Let $Q$ be an Euclidean quiver. Let $V$ with $Ext_{Q}^{1}(V,V)\neq 0$ be a stable module with respect to a central charge $Z:K(mod-kQ)\rightarrow\mathbb{C}$ on $mod-kQ$. Then there is no stable preinjective/postprojective indecomposable with central charge $Z(X)$ right/left to the ray $Z(\delta)$ in the upper halfplane $\overline{\mathbb{H}}$.
\end{lem}
\begin{proof}
A stable module is a brick. By Lemma \ref{achscheisse} and the proof of Lemma \ref{achscheisse2} we have $\underline{dim}\ V=\delta$. Let us assume there is a stable postprojective indecomposable $X$ with $Z(X)$ left to $Z(V)$ in $\overline{\mathbb{H}}$. By formula (\ref{formula}) and Proposition \ref{dlabringel} we have 
\begin{align*}
-dim\ Hom_{Q}(X,V)+dim\ Ext_{Q}^{1}(X,V)&=-\left\langle \underline{dim}\ X,\underline{dim}\ V\right\rangle\\
&=-\left\langle \underline{dim}\ X,\delta\right\rangle\\
&=\left\langle \delta,\underline{dim}\ X\right\rangle<0. 
\end{align*}
Since $\phi(X)>\phi(V)$ and $X$ and $V$ are stable we have $Hom_{Q}(X,V)=0$. This is a contradiction. The claim for the stable preinjective indecomposable follows by similar arguments.
\end{proof}

Let us consider the Kronecker quiver in the setting above again. All indecomposables in the familiy $(\ref{kronecker})$ are stable. Thus we have infinitely many stable modules $X$ with $\underline{dim}\ X=\delta=(1,1)$ and $Ext_{Q}^{1}(X,X)=k$. The postprojective respectively preinjective modules are the modules with dimension vectors $(m,m+1)$ respectively $(m+1,m)$. These are precisely the exceptional modules of the Kronecker quiver. By Lemma (\ref{achscheisse}) the dimension vectors of all stable modules with self-extensions are multiples of $(1,1)$. We see that stable modules without self-extensions have a unique phase, i.e. all other stable modules have a different phase.\\ 

Since $mod-kQ$ is an algebraic heart in $D^{b}(mod-kQ)$ we can apply the categorical mutation method from section 3 in this case for an Euclidean or more general for an acyclic quiver.  

\begin{prop}
\label{qeuclidean}
Let $Q$ be an Euclidean quiver. Let $Z:K(mod-kQ)\rightarrow\mathbb{C}$ be a rigid central charge on $mod-kQ$. Further we assume there is a stable module $V$ with $Ext_{Q}^{1}(V,V)\neq 0$. Then there are infinitely many stable postprojective respectively preinjective stable indecomposables approaching the ray $Z(\delta)$ from right respectively left in the upper halfplane $\overline{\mathbb{H}}$.
\end{prop}                  
\begin{proof}
By Theorem 5.7 in \cite{70} all hearts obtained by a finite sequence of simple tilts from $mod-kQ$ in $D^{b}(mod-kQ)$ are algebraic and rigid. By Proposition \ref{prop2} and Lemma \ref{achscheisse} there are infinitely many stable modules whose central charges lie on the left/right of the ray $Z(\delta)$ approaching a limit point. By Lemma 3.15 in \cite{210} if there are infinitely many stable modules for a Euclidean quiver their phases approach an unique limit point. There are only finitely many exceptional regular modules by Lemma \ref{achscheisse2} and thus we have infinitely many stable preinjective/postprojective indecomposables left/right to the ray $Z(\delta)$ by Lemma \ref{achscheisse3}.\end{proof}

\begin{rem}
By Lemma \ref{achscheisse} the conditions in Proposition $\ref{qeuclidean}$ on the central charges can be checked by calculating the real and imaginary roots. Note that the phases of stable modules are not dense in this case.
\end{rem}

Let us check Proposition \ref{qeuclidean} in the case of the Kronecker quiver: The simple $S_2=0\rightrightarrows k$ is projective and the simple $S_1=k\rightrightarrows 0$ is injective. For $\phi(S_2)>\phi(S_1)$ there can be no stable module with self-extensions. Indeed, in this case by the mutation method from section 2 the only stable modules are $S_1$ and $S_2$ since $Ext^{1}_{Q}(S_2,S_1)=0$ (see Lemma 8.2.1 in \cite{28}).\\

Given a quiver $Q$, let $Stab(Q)$ be the space of locally-finite stability conditions on $D^{b}(mod-kQ)$.

\begin{prop}
\label{prop3}
Let $Q$ be the Kronecker quiver (\ref{kroneckerquiver}) and let the heart $\mathcal{A}=\mathcal{P}((0,1])$ of a stability condition $\sigma=(Z,\mathcal{P})\in Stab(Q)$ be reachable by a finite sequence of simple tilts from the heart $mod-kQ$. Let the two simple objects of $\mathcal{A}$ have different phases. If there is a stable object $A\in \mathcal{P}(\phi)$ for $\phi\in (0,1]$ with $Ext_{\mathcal{A}}^{1}(A,A)\neq 0$, then $A=A'[i]$ where $A'$ is a module in the $\mathbb{P}_{k}^{1}$-family (\ref{kronecker}), $i\in\mathbb{Z}$ and there are infinitely many stable objects approaching the ray $Z(A)$ from the left respectively right. Further, all or almost all of the stable objects approaching the ray $Z(A)$ from the left/right are of the form $E[j]$ where $E$ is a preinjective/postprojective indecomposable in $mod-kQ$ and $j\in\mathbb{Z}$.  
\end{prop}
\begin{proof}
The stable objects $E\in\mathcal{P}(\phi)$ for some $\phi\in\mathbb{R}$ are the simple objects and thus indecomposable in the Abelian category $\mathcal{P}(\phi)$. Thus there are no idempotents except $0$ and $id$ in $End_{\mathcal{P}(\phi)}(E)=End_{D^{b}(mod-kQ)}(E)$ and thus the $\sigma$-stable objects are indecomposable in $D^{b}(mod-kQ)$. Since $mod-kQ$ is hereditary $E$ is of the form $E=E'[i]$ for some indecomposable $E'\in mod-kQ$ and $i\in\mathbb{Z}$. In particular, the stable object $A$ is of the form $A'[i]$ with $A'$ indecomposable with $Ext^{1}_{Q}(A',A')\neq 0$ and $Hom_Q(A',A')=k$ by the generalized Schur lemma \cite{260}. Thus $A'$ is a regular brick with $\underline{dim}\ A'=\delta=(1,1)$, i.e. an element of the family (\ref{kronecker}). Note that by the 2. and 3. axiom of Definition \ref{bridgeland} stable objects $E$ and $E[i]$ for $i\in\mathbb{Z}$ do not lie in the same subcategory $\mathcal{P}(\phi)$. By the discussion of the Kronecker quiver above follows that all stable objects in $\mathcal{P}((0,1])$ without self-extensions have a unique phase in $(0,1]$. By Proposition \ref{prop2} we have infinitely many stable objects without self-extensions in $\mathcal{A}$ approaching the ray $Z(A)$ from left and right. From Lemma 3.15 in \cite{210} follows that the phases of stable objects approach a unique limit point from the left and the right. These stable objects are of the form $E'[i]$ for some indecomposable $E'\in mod-kQ$ and $i\in\mathbb{Z}$ with $Ext^{1}_{Q}(E',E')=0$, i.e. $E'$ is a preinjective or postprojective indecomposable. Therefore all or almost all stable objects approaching the ray $Z(E)$ from the left/right must be of the form $E'[i]$ with $E'$ preinjective/postprojective indecomposable. 
\end{proof}

Let us consider $mod-kQ$ for the Kronecker quiver (\ref{kroneckerquiver}) with central charge given by the phases $\phi(S_1)>\phi(S_2)$. If we rotate the central charge a bit counter-clockwise, the stable modules of $mod-kQ$ are stable in the tilt $\mathcal{A}$ of $mod-kQ$ at the simple $S_1$ again. Note that the right-most object is now $S_1[-1]$ where $S_1$ is injective. Proposition \ref{prop3} is more general: It is a statement on hearts obtained from sequences of simple tilts from $mod-kQ$ not necessary induced by a central charge.

\section{Stable representations of wild quivers}

An acyclic quiver $Q$ is \textit{wild} if it is neither of Dynkin nor of Euclidean type. An example is the generalized Kronecker quiver with $l\geq 3$ arrows:
\begin{align}
\label{kl}
\xymatrix{1 \ar@/^1pc/[r]\ar@/_/[r]^{\vdots\ l} & 2}, \hspace{0.5cm}l\geq 3.
\end{align}

We define a central charge by choosing two complex numbers $Z(S_1)=r_1\exp(i\pi\phi_1)$ and $Z(S_2)=r_2\exp(i\pi\phi_2)$ with $0<\phi_2<\phi_1<1$ for the simple modules $S_1, S_2$ associated to the two vertices. By Lemma 3.18 in \cite{210} the phases of the central charges of the positive real roots approach $\arg\ Z\left(\frac{1}{2}\left(l+\sqrt{l^2+4}\right)\right)$ from the left, respectively $\arg\ Z\left(\frac{1}{2}\left(l-\sqrt{l^2+4}\right)\right)$ from the right. Note that the Coxeter transformation of the Kronecker quiver has eigenvectors $\frac{1}{2}\left(l\pm\sqrt{l^2+4}\right)$. The phases of the positive imaginary roots lie dense in the arc between these two phases. In particular, the postprojective and preinjective modules have unique phases. The modules of the $\mathbb{P}_{k}^{l-1}$-family of indecomposable modules of dimension vector $(1,1)$ are stable. These have self-extensions thus there are infinitely many stable postprojective and preinjective modules by Proposition \ref{prop2}. In fact, all postprojectives and preinjectives are stable. Moreover, if we restrict to the field $k=\mathbb{C}$, for any positive root $\alpha$ there is a semistable module $X$ with $\underline{dim}\ X=\alpha$ by Corollary 3.20 in \cite{210}. This implies for any phase $\phi$ of an indecomposable module there is a stable module with this phase since the stable modules are the simple objects in the exact subcategory of semistable modules of phase $\phi$. Therefore there is an arc where the phases of the stable (regular) modules are dense and outside we find the phases of the exceptional modules, i.e. the postprojective and preinjective modules. For stable representations of the generalized Kronecker quiver compare \cite{310}.

\begin{rem}
For the generalized Kronecker quiver \ref{kl} an analogue of Proposition \ref{prop3} holds true.
\end{rem}

As another illustrative example let us consider the following quiver:

$$\xymatrix{1 \ar@<+.8ex>[r]\ar[r]\ar@<-.7ex>[r] & 2\ar[r]&3}$$

This quiver has the generalized Kronecker quiver with three arrows (or more generally with $l\geq 3$ arrows) as a full subquiver. Let us choose the central charge given by three complex numbers $Z(S_1), Z(S_2)$ and $Z(S_3)$ in the upper halfplane $\overline{\mathbb{H}}$ with phases $\phi(S_3)>\phi(S_1)>\phi(S_2)$. With respect to this central charge we have the stable representations induced by the stable representations of the 3-Kronecker quiver with central charge given by $Z(S_1)$ and $Z(S_2)$. Representations supported on the vertices $1,2,3$ or $2,3$ are not stable. Thus the only stable representation not coming from the 3-Kronecker quiver is $S_3$ and this central charge is rigid.\\
We call a module $X$ \textit{sincere} if each component of $\underline{dim}\ X$ is non-zero. Since there are only finitely many non-sincere postprojective or preinjective indecomposables as follows from Theorem \ref{delapena} we have in $mod-kQ$ (where $k$ is a algebraically closed field) infinitely many non-sincere regular stable modules whose phases approach two limit points. Since every regular component contains only finitely many non-sincere modules (see Corollary 2.4 in chapter XVII of \cite{300}) at most finitely many of these stable modules lie in a regular component.\\

If we take the quiver $$\xymatrix{1 \ar@<+.7ex>[r]\ar@<-.7ex>[r] & 2 \ar[r]&3}$$ with subquiver given by the Kronecker quiver and again choose a central charge with $\phi(S_3)>\phi(S_1)>\phi(S_2)$ the infinitely many stable modules induced from the Kronecker quiver approach an unique limit point.\\

If we apply the mutation method of section 3 to a quiver $Q$ the simple objects of the hearts appearing in the mutation algorithm lying in $mod-kQ$ are indecomposable modules in $mod-kQ$. Thus they are postprojective, preinjective or regular.

\begin{prop}
Let $Q$ be an acyclic non-Dynkin quiver with a rigid central charge on $mod-kQ$. If the left-most stable object $S$ of a heart $\mathcal{A}'$ appearing in the mutation method is a postprojective/preinjective/regular indecomposable in $mod-kQ$ and there are except $S$ $l$ further simple objects that are postprojectvie/preinjective/regular in $mod-kQ$, then there are $l$ simple objects in the heart $\mathcal{A}'_{S}$ obtained by a left-tilt at the simple $S$ that are postprojectvie/preinjective/regular in $mod-kQ$. 
\end{prop}
\begin{proof}
By the formulas in Proposition 5.2 in \cite{70} we can calculate the new simple objects of $\mathcal{A}'_{S}$ from the old simple objects. For the classes of the $l$ simple objects $S_1,\ldots,S_l$ we have $[S'_j]=[S_j]+m[S]$ where $m=\dim\ Ext^1(S,S_j)$ and $S'_j$ is the new simple object calculated from the simple $S_j$. Since $m\geq 0$ the result follows from Theorem \ref{dlabringel} and Theorem \ref{delapena}.  
\end{proof}

Let $Q$ be a wild quiver with $n$ vertices. Let $\Phi:\mathbb{Z}Q_0\rightarrow \mathbb{Z}Q_0$ be the Coxeter transformation of $Q$ defined in the last section. We consider $\Phi$ as an endomorphism of $\mathbb{R}^{n}=\mathbb{Z}Q_0\otimes\mathbb{R}$. The \textit{spectrum} $Spec(\Phi)$ is the set of eigenvalues of $\Phi$ and the \textit{spectral radius} $\rho(\Phi)$ is the positive real number $$\rho(\Phi)=max\left\{|\lambda|\ |\lambda\in\ Spec(\Phi)\right\}.$$ $\rho(\Phi)$ is called the \textit{growth number} of $Q$ due to the statement of Theorem \ref{delapena}. 

\begin{thm}\cite{270}
Let $Q$ be a wild quiver with Coxeter transformation $\Phi$. Then
\begin{enumerate}
\item[(i)] $\rho(\Phi)>1$ and $\rho(\Phi)$ is an eigenvalue of $\Phi$ of multiplicity one.
\item[(ii)] For any other eigenvalue $\mu\in\ Spec(\Phi)$ we have $|\mu|<\rho(\Phi)$.
\item[(iii)] There is a strictly positive\footnote{A vector $y\in \mathbb{R}^{n}$ is \textit{strictly positive}, if all its coordinates are positive.} eigenvector $y^{+}$ of $\Phi$ with $\Phi y^{+}=\rho(\Phi)y^{+}$ and a strictly positive eigenvector $y^{-}$ of $\Phi^{-1}$ with $\Phi y^{-}=\rho(\Phi)^{-1}y^{-}$.
\end{enumerate}
\end{thm}             

Note that the vectors $y^+$ and $y^-$ are linearly independent. We have the following analogue of Theorem \ref{dlabringel} in the wild case:

\begin{thm}\cite{280}
\label{delapena}
Let $Q$ be a wild quiver. Let $X\in mod-kQ$ be indecomposable. Then $X$ is
\begin{enumerate}
\item[(i)] postprojective if and only if $\left\langle y^{-}, \underline{dim}\ X\right\rangle<0$. Moreover, if $X$ is not postprojective, then $\left\langle y^{-}, \underline{dim}\ X\right\rangle>0$.
\item[(ii)] preinjective if and only if $\left\langle \underline{dim}\ X,y^{+}\right\rangle<0$. Moreover, if $X$ is not preinjective, then $\left\langle \underline{dim}\ X,y^{+}\right\rangle>0$.
\item[(iii)] regular if and only if $\left\langle y^{-}, \underline{dim}\ X\right\rangle>0$ and $\left\langle \underline{dim}\ X,y^{+}\right\rangle>0$.
\end{enumerate}
\end{thm}

An important consequence is the following Theorem that allows to describe the long-term behaviour of the discrete dynamical system defined by the Coxeter transformation $\Phi$ with initial conditions given by the dimension vectors of indecomposable modules.

\begin{thm}\cite{280, 290}
\label{delapena2}
Let $Q$ be a wild quiver and let $X\in mod-kQ$ be an indecomposable module. Then
\begin{enumerate}
\item[(i)] If $X$ is postprojective or regular, then $$\lim_{m\rightarrow\infty}\frac{1}{\rho(\Phi)}\underline{dim}\ \tau^{-m}X=\lambda^{-}_{X}y^{-}$$ for some $\lambda^{-}_{X}>0$.
\item[(ii)] If $X$ is preinjective or regular, then $$\lim_{m\rightarrow\infty}\frac{1}{\rho(\Phi)}\underline{dim}\ \tau^{m}X=\lambda^{+}_{X}y^{+}$$ for some $\lambda^{+}_{X}>0$.
\end{enumerate}
\end{thm} 

\begin{cor}
\label{centralcharge}
Let $Q$ be a wild quiver and let $X\in mod-kQ$ be an indecomposable module. Given a central charge $Z:K(mod-kQ)\rightarrow\mathbb{C}$ then
\begin{enumerate}
\item[(i)] If $X$ is postprojective or regular, then $\arg\ Z(\underline{dim}\ \tau^{-m}X)\rightarrow \arg\ Z(y^{-})$ for $m\rightarrow\infty$.
\item[(ii)] If $X$ is preinjective or regular, then $\arg\ Z(\underline{dim}\ \tau^{m}X)\rightarrow \arg\ Z(y^{+})$ for $m\rightarrow\infty$.
\end{enumerate}
\end{cor}
\begin{proof}
Let $X$ be postprojective or regular. Note first that $Z(y^{-})$ lies in the upper halfplane $\overline{\mathbb{H}}$ since $y^{-}$ is a strictly positive vector. We consider the linear map of real vector spaces $Z:K(mod-kQ)\otimes_{\mathbb{Z}}\mathbb{R}\rightarrow\mathbb{C}$. The sequence of vectors $\frac{1}{\rho(\Phi)}Z(\underline{dim}\ \tau^{-m}X)$ converge to $\lambda^{-}_{X}Z(y^{-})$ for $m\rightarrow\infty$. Thus the phases of the complex numbers $\frac{1}{\rho(\Phi)}Z(\underline{dim}\ \tau^{-m}X)$ and $\lambda^{-}_{X}Z(y^{-})$ become arbitrary close. The other case follows similarly.
\end{proof}

We have the first observation for the phases of stable modules:

\begin{lem}
Let $Q$ be a wild quiver. We assume we have a central charge on $mod-kQ$. Then the phases of at most finitely many stable indecomposable modules of a regular component $\mathcal{C}$ lie right/left to a stable postprojective/preinjective indecomposable modules.           
\end{lem}
\begin{proof}
By XVIII.2, 2.5 Corollary in \cite{300} for all but finitely many modules $X$ from $\mathcal{C}$ we have $Hom_{Q}(P,X)\neq 0$ for a postprojective module $P$ and $Hom_{Q}(X,I)\neq 0$ for a postprojective module $I$.    
\end{proof}

The following Proposition is the main result of this section:

\begin{prop}
\label{wild}
Let $Q$ be a wild quiver. Keep the notation from above. Let $Z:K(mod-kQ)\rightarrow\mathbb{C}$ be a rigid central charge on $mod-kQ$. Further we assume there is a stable module $V$ with $Ext_{Q}^{1}(V,V)\neq 0$. Then:
\begin{enumerate}
\item[(i)] There are infinitely many stable exceptional modules whose phases approach a limit point from the left respectively a limit point from the right in the upper halfplane $\overline{\mathbb{H}}$.
\item[(ii)] The infinitely many stable exceptional modules from (i) contain at most finitely many postprojective/preinjective indecomposable modules on the left/right.
\item[(iii)] If the infinitely many stable exceptional modules from (i) contain infinitely many postprojective and preinjective modules, the central charges of infinitely many stable postprojective/preinjective modules lie right/left from $Z(V)$ and approach $\arg\ Z(y^-)/\arg\ Z(y^+)$. Further, the central charges of the stable regular modules lie left to $Z(y^-)$ and right to $Z(y^+)$.
\item[(iv)] If the infinitely many stable exceptional modules from (i) do not approach $\arg\ Z(y^+)$ or $\arg\ Z(y^-)$ from left or right, then they contain infinitely many regular modules and, moreover, at most finitely many modules from each regular component.
\end{enumerate}     
\end{prop}
\begin{proof}
(i) The first part follows from Proposition \ref{prop2}.\\

(ii) If we have infinitely many stable postprojective indecomposables, then we can find a postprojective module of the form $\tau^{-m}X$ with $\tau$ the Auslander-Reiten functor and $X$ postprojective for arbitrary large $m>0$. Let $Y$ be a stable regular indecomposable. For a postprojective indecomposable $X$ we have $$\lim_{m\rightarrow\infty}\frac{1}{\rho(\Phi)^m}\left\langle \underline{dim}\ \tau^{-m}X, \underline{dim}\ Y\right\rangle=\lambda_{X}^{-}\left\langle y^-,\underline{dim}\ Y \right\rangle>0$$ by Theorem \ref{delapena} and Theorem \ref{delapena2}. Thus $Hom_{Q}(\tau^{-m}X,Y)\neq 0$ for $m\gg 0$ and $\phi(\tau^{-m}X)<\phi(Y)$ for $m$ large enough. Thus the regular $V$ with $Ext_{Q}^{1}(V,V)\neq 0$ lies to the left of some of the postprojective indecomposables approaching a limit point from the left. But this is a contradiction since all stable modules with phases left to this limit point are exceptional. The statement for the stable preinjective indecomposables on the right is proven similarly.\\

(iii) This follows from Corollary \ref{centralcharge}.\\

(iv) We consider the infinitely many stable exceptional modules from (i). If we could find among these a stable module of the form $\tau^{m}X$ or $\tau^{-m}X$ with $X$ indecomposable and $m>0$ arbitrary large, then these stable modules would approach $\arg\ Z(y^+)$ or $\arg\ Z(y^-)$. Thus the stable exceptionales from (i) contain at most finitely many postprojective or preinjective modules. Each regular component $\mathcal{C}$ is of the type $\mathbb{Z}A_{\infty}$ and consists of the modules $\tau^{m}X[i]$ with $m\in\mathbb{Z}$, $i\geq 1$ and $X[i]$ is defined by the infinite chain of irreducible injective morphisms $$X=X[1]\rightarrow X[2]\rightarrow \cdots \rightarrow X[i]\rightarrow\cdots$$ where $X$ is a quasi-simple regular module in $\mathcal{C}$ (see chapter XVIII of \cite{300}). If the module $\tau^{m}X$ is exceptional then the modules in its $\tau$-orbit are exceptional. By Corollary 2.16 in chapter XVIII of \cite{300} if $X[i]$ is exceptional, then $i\leq \#Q_0-2$. Thus at most finitely many stable modules described in (i) lie in the regular component $\mathcal{C}$.   
\end{proof}

\section*{Acknowledgements} 
I thank Lutz Hille, Markus Reineke, Jan Schr\"oer and Michele del Zotto for useful discussions and/or correspondences. Further I am grateful to the Hausdorff Center for Mathematics of the University of Bonn for support.

\end{document}